\documentclass[11pt]{article}
       \usepackage{amsfonts}
       \usepackage{stmaryrd}
       \usepackage{latexsym,amssymb,mathrsfs,fancyhdr}
       \font\tenmsb=msbm10
       \font\sevenmsb=msbm7
       \font\fivemsb=msbm5
       \catcode`\@=11
       \ifx\amstexloaded@\relax\catcode`\@=\active
       \endinput\else\let\amstexloaded@\relax\fi
       \def\spaces@{\space\space\space\space\space}
       \def\spaces@@{\spaces@\spaces@\spaces@\spaces@\spaces@}
       \def\space@.  {\futurelet\space@\relax}
       \space@.   %
       \def\Err@#1{\errhelp\defaulthelp@\errmessage{AmS-teX error: #1}}
       \def\relaxnext@{\let\next\relax}
       \def\accentfam@{7}
       \def\noaccents@{\def\accentfam@{0}}
       \def\Cal{\relaxnext@\ifmmode\let\next\Cal@\else
       \def\next{\Err@{Use \string\Cal\space only in math mode}}\fi\next}
       \def\Cal@#1{{\Cal@@{#1}}}
       \def\Cal@@#1{\noaccents@\fam\tw@#1}
       \def\Bbb{\relaxnext@\ifmmode\let\next\Bbb@\else
       \def\next{\Err@{Use \string\Bbb\space only in math mode}}\fi\next}
       \def\Bbb@#1{{\Bbb@@{#1}}}
       \def\Bbb@@#1{\noaccents@\fam\msbfam#1}
       \def\co{\tiny{\textcircled{\tiny\#}}}
       \newfam\msbfam
       \textfont\msbfam=\tenmsb
       \scriptfont\msbfam=\sevenmsb
       \scriptscriptfont\msbfam=\fivemsb
\usepackage{dsfont}
\usepackage[german,english]{babel}
\usepackage{amsmath,amssymb}
\usepackage[square, comma, sort&compress, numbers]{natbib}
\usepackage{cases}
\usepackage{mathrsfs}
\usepackage{amsmath}
\usepackage{amsthm}
\usepackage{amsfonts}
\usepackage{amssymb}
\usepackage{latexsym}
\usepackage{fancyhdr}
\usepackage{extarrows}

\newtheorem{thm}{Theorem}[section]
\newtheorem{prop}[thm]{Proposition}
\newtheorem{lem}[thm]{Lemma}
\newtheorem{rem}[thm]{Remark}
\newtheorem{iteration lemma}[thm]{iteration Lemma}
\newtheorem{cor}[thm]{Corollary}

\newtheorem{eg}[thm]{Example}

\newtheorem*{acknowledgements*}{ACKNOWLEDGEMENtS}

\begin{document}

\setlength{\columnsep}{5pt}
\title{\bf Characterizations of core and dual core inverses in rings with involution}
\author{Tingting  Li\footnote{ E-mail: littnanjing@163.com},
\ Jianlong Chen\footnote{ Corresponding author. E-mail: jlchen@seu.edu.cn} \\
Department of  Mathematics, Southeast University \\  Nanjing 210096,  China }
     \date{}

\maketitle
\begin{quote}
{\textbf{}\small
Let $R$ be a unital ring with involution,
we give the characterizations and representations of the core and dual core inverses of an element in $R$ by Hermitian elements (or projections) and units.
For example,
let $a\in R$ and $n\geqslant 1$,
then $a$ is core invertible if and only if there exists a Hermitian element (or a projection) $p$ such that  $pa=0$, $a^{n}+p$ is invertible.
As a consequence,
$a$ is an $\mathrm{EP}$ element if and only if there exists a Hermitian element (or a projection) $p$ such that  $pa=ap=0$, $a^{n}+p$ is invertible.
We also get a new characterization for both core invertible and dual core invertible of a regular element by units,
and their expressions are shown.
In particular,
we prove that for $n\geqslant 2$,
$a$ is both Moore-Penrose invertible and group invertible if and only if $(a^{\ast})^{n}$ is invertible along $a$.


}
\end{quote}

\section{ Introduction }\label{a}
Let $R$ be a unital ring.
Recall that an element $a\in R$ is (von Neumann) regular if there exists $x\in R$ satisfying $axa=a$.
Such $x$ is called an inner inverse of $a$ and denoted by $a^{-}$.
If there is an element $x\in R$ such that $axa=a$ and $xax=x$,
then $x$ is called the $\{1,2\}-$inverse of $a$.
The symbols $a\{1\}$ and $a\{1,2\}$ denote the set of all inner inverse and $\{1,2\}-$inverse of $a$,
respectively.
An element $a\in R$ is group invertible if there is $x\in a\{1,2\}$ that commutes with $a$.
The group inverse of $a$ is unique if it exists and denoted by $a^{\#}$.

An involution $\ast$ in $R$ is an anti-isomorphism of degree 2 in $R$,
that is to say,
$(a^{\ast})^{\ast}=a, (a+b)^{\ast}=a^{\ast}+b^{\ast}$ and $(ab)^{\ast}=b^{\ast}a^{\ast}$ for all $a, b\in R$.
We say that $a$ is Moore-Penrose invertible if there exists $x$ satisfying the following four equations:
\begin{equation*}
\begin{split}
    (1)~axa=a,~(2)~xax=x,~(3)~(ax)^{\ast}=ax,~(4)~(xa)^{\ast}=xa.
\end{split}
\end{equation*}
If such $x$ exists,
then it is called a Moore-Penrose inverse of $a$,
it is unique and denoted by $a^{\dagger}$.
If $x$ satisfies the equations $(1)$ and $(3)$,
then $x$ is called a $\{1,3\}$-inverse of $a$ and denoted by $a^{(1,3)}$,
and a $\{1,4\}$-inverse of $a$ can be similarly defined.
If $a$ is both Moore-Penrose invertible and group invertible with $a^{\dagger}=a^{\#}$,
then $a$ is said to be an EP element.
The set of all Moore-Penrose invertible,
group invertible,
invertible,
$\{1,3\}$-invertible,
$\{1,4\}$-invertible and EP elements in $R$ are denoted by the symbols $R^{\dagger}, R^{\#}, R^{-1}, R^{\{1,3\}}$, $R^{\{1,4\}}$ and $R^{\mathrm{EP}}$,
respectively.
An element $a$ is Hermitian (or symmetric) if $a^{\ast}=a$,
and $a$ is called an idempotent if $a^{2}=a$.
A Hermitian idempotent is said to be a projection.

The core and dual core inverses of a complex matrix were introduced by Baksalary and Trenkler in \cite{OM}.
Raki\'{c} et al. \cite{DSR} generalized core inverses of a complex matrix to the case of an element in a ring with involution.
An element $x\in R$ is said to be a core inverse of $a$ if it satisfies
\begin{equation*}
\begin{split}
    axa=a,~xR=aR,~Rx=Ra^{\ast},
\end{split}
\end{equation*}
such an element $x$ is unique if it exists and denoted by $a^{\co}$.
There is a dual concept of core inverses which is called dual core inverses.
The symbols $R^{\co}$ and $R_{\co}$ stand for the set of all the core invertible and dual core invertible elements in $R$, respectively.

In \cite{KPS},
K.P.S. Bhaskara Rao characterized group inverse of an element by idempotents.
In \cite{Han},
R.Z. Han and J.L. Chen characterized $\{1,3\}-$inverse of an element by projections.
In \cite{XSZ2},
S.Z. Xu et al. gave the characterizations and expressions of Moore-Penrose inverse of an element by a Hermitian element (or a projection).
Motivated by the above mentioned results,
this paper characterizes the existence of core and dual core inverses of an element by Hermitian elements (or projections) and units.
Moreover,
we give their representations.
We prove that $a$ is core invertible if and only if there exists a Hermitian element (or a projection) $p$ such that  $pa=0$, $a^{n}+p$ is invertible for $n\geqslant 1$.
As a consequence,
we get a new characterization of an $\mathrm{EP}$ element,
namely,
$a$ is an $\mathrm{EP}$ element if and only if there exists a Hermitian element (or a projection) $p$ such that  $pa=ap=0$, $a^{n}+p$ is invertible.

In \cite{PP},
P. Patri\'{c}io et al. discussed the characterizations and expressions of Moore-Penrose inverse of a regular element.
In \cite{Chen},
J.L. Chen et al. gave the characterizations and expressions of group inverse and core inverse of a regular element,
and they also showed the characterizations for both core invertible and dual core invertible of a regular element.
This article give a new characterization for both core invertible and dual core invertible of a regular element.

In \cite{X3},
X. Mary introduced a new generalized inverse,
namely,
if $a, d\in R$ and there exists $y$ such that $yad=d=day$, $yR\subset dR$ and $Ry\subset Rd$,
then we say that $a$ is invertible along $d$. 
In which case, 
$y$ is called the inverse of $a$ along $d$ and denoted by $a^{\|d}$.
In \cite{CZ},
H.H. Zhu et al. said that $a$ is left (resp. right) invertible along $d$ if there exists $y$ such that $yad=d$ (resp. $day=d$) and $Ry\subset Rd$ (resp. $yR\subset dR$).
X. Mary\cite{X3} told us that $a^{\#}=a^{\|a}$ and $a^{\dagger}=a^{\|a^{\ast}}$.
It is easy to know that $a\in R^{\dagger}$ if and only if $a^{\ast}$ is invertible along $a$, 
this paper will prove that for $n\geqslant 2$,
$a\in R^{\dagger}\cap R^{\#}$ if and only if $(a^{\ast})^{n}$ is invertible along $a$.

We will also use the following notations:
$aR=\{ax | x\in R\}$, $Ra=\{xa | x\in R\}$, $^{\circ}\!a=\{x\in R | xa=0\}$, $a^{\circ}=\{x\in R | ax=0\}$.

\section{Preliminaries}\label{a}
Throughout this paper,
$R$ is a unital ring with involution.
In this section,
some auxiliary lemmas and results are presented for the further reference.

\begin{lem} \cite[p.$201$]{H}\label{13-14-inverse}
Let $a\in R$,
we have the following results:\\
$(1)$ $a$ is $\{1,3\}$-invertible with $\{1,3\}$-inverse $x$ if and only if $x^{\ast}a^{\ast}a=a;$\\
$(2)$ $a$ is $\{1,4\}$-invertible with $\{1,4\}$-inverse $y$ if and only if $aa^{\ast}y^{\ast}=a.$
\end{lem}

This lemma tells us that $a\in R^{\{1,3\}}$ if and only if $a\in Ra^{\ast}a$.
Similarly,
$a\in R^{\{1,4\}}$ if and only if $a\in aa^{\ast}R$.

\begin{lem}(\cite[Theorem $2.19$ and $2.20$]{CZ}, \cite{P})\label{M-P1}
Let $a\in R$,
the following conditions are equivalent:\\
$(1)$ $a\in R^{\dagger}$;\\
$(2)$ $a\in Raa^{\ast}a$;\\
$(3)$ $a\in aa^{\ast}aR$;\\
$(4)$ $a\in Ra^{\ast}a\cap aa^{\ast}R$.\\
In this case,
\begin{equation*}
\begin{split}
    a^{\dagger}= a^{\ast}ax^{2}a^{\ast}= a^{\ast}y^{2}aa^{\ast}= sat,
\end{split}
\end{equation*}
where $a= aa^{\ast}ax= yaa^{\ast}a= t^{\ast}a^{\ast}a= aa^{\ast}s^{\ast}$.
\end{lem}

The condition $(4)$ of this lemma shows that $a\in R^{\dagger}$ if and only if $a\in R^{\{1,3\}}\cap R^{\{1,4\}}$.

\begin{lem} \cite[Theorem $2.19$]{DSR}\label{core->group}
Let $a\in R$.\\
$(1)$ If $a\in R^{\co}$,
then $a\in R^{\#}$ and $a^{\#}=(a^{\co})^{2}a$;\\
$(2)$ if $a\in R_{\co}$,
then $a\in R^{\#}$ and $a^{\#}=a(a_{\co})^{2}$.
\end{lem}

\begin{lem} \cite[Theorem $2.6$ and $2.8$]{XSZ}\label{dual-core-inverse}
Let $a\in R$,
we have the following results:\\
$(1)$ $a\in R^{\co}$ if and only if $a\in R^{\#}\cap R^{\{1,3\}}$. In this case, $a^{\co}=a^{\#}aa^{(1,3)}$.\\
$(2)$ $a\in R_{\co}$ if and only if $a\in R^{\#}\cap R^{\{1,4\}}$. In this case, $a_{\co}=a^{(1,4)}aa^{\#}$.
\end{lem}

\begin{lem} \cite[Proposition $7$]{H}\label{group-inverse}
Let $a\in R$,
$a\in R^{\#}$ if and only if $a=a^{2}x=ya^{2}$ for some $x, y\in R$.
In this case, $a^{\#}=yax=y^{2}a=ax^{2}$.
\end{lem}

\begin{lem} \cite{BAB, DSC}\label{J.L}
Let $a, b\in R$,
$1+ab$ is invertible if and only if $1+ba$ is invertible.
\end{lem}

\begin{lem} (\cite[Theorem $3.2$]{X1}and \cite[Theorem $1.3$]{X2}) \label{Mary}
Let $d\in R$ is regular with $d^{-}\in d{\{1\}}$.
The following conditions are equivalent:\\
$(1)$ $a$ is invertible along $d$;\\
$(2)$ $u=da+1-dd^{-}$ is invertible;\\
$(3)$ $v=ad+1-d^{-}d$ is invertible.\\
In this case,
$a^{\|d}=u^{-1}d=dv^{-1}$.
\end{lem}

\begin{lem} \cite[Theorem $2.3$ and Corollary $3.3$]{CZ} \label{ZHH}
Let $d\in R$ is regular with $d^{-}\in d{\{1\}}$.
The following conditions are equivalent:\\
$(1)$ $a$ is left (resp. right) invertible along $d$;\\
$(2)$ $d\in Rdad$ (resp. $d\in dadR$);\\
$(3)$ $da+1-dd^{-}$ is left (resp. right) invertible;\\
$(4)$ $ad+1-d^{-}d$ is left (resp. right) invertible.\\
\end{lem}

Hartwig et al.\cite[Theorem $1$]{H2} proved that $a\in R^{\#}$ if and only if $a\in Ra^{2}\cap a^{2}R$,
so we obtain that $a\in R^{\co}$ if and only if $a\in Ra^{\ast}a\cap Ra^{2}\cap a^{2}R$ by Lemma~\ref{13-14-inverse} and Lemma~\ref{dual-core-inverse}.
We aim at characterizing the core invertibility by the intersection of two left principal ideals,
let us start with a simple proposition.

\begin{prop} \label{a**a}
Let $a\in R$,
$n\geqslant 1$,
we have the following results:\\
(I) the following conditions are equivalent:\\
\indent $(1)$ $a\in R(a^{\ast})^{n}a$;\\
\indent $(2)$ $a\in Ra^{\ast}a\cap a^{n}R$;\\
\indent $(3)$ $R=$ $^{\circ}\!a\oplus R(a^{\ast})^{n}$;\\
\indent $(4)$ $R=$ $^{\circ}\!a+R(a^{\ast})^{n}$;\\
(II) the following conditions are equivalent:\\
\indent $(1)$ $a\in a(a^{\ast})^{n}R$;\\
\indent $(2)$ $a\in aa^{\ast}R\cap Ra^{n}$;\\
\indent $(3)$ $R=a^{\circ}\oplus (a^{\ast})^{n}R$;\\
\indent $(4)$ $R=a^{\circ}+(a^{\ast})^{n}R$.
\end{prop}

\begin{proof}
(I)
$(1)\Rightarrow (2).$ It is clear to see that $a\in Ra^{\ast}a$ follows from $a\in R(a^{\ast})^{n}a$,
and there exists $r\in R$ such that
\begin{equation*}
\begin{split}
    a=r(a^{\ast})^{n}a=r(a^{\ast})^{n-1}a^{\ast}a,
\end{split}
\end{equation*}
so $a\in R^{\{1,3\}}$ and $(r(a^{\ast})^{n-1})^{\ast}=a^{n-1}r^{\ast}\in a\{1,3\}$ by Lemma~\ref{13-14-inverse}.
Moreover,
\begin{equation*}
\begin{split}
    a=aa^{(1,3)}a=a(a^{n-1}r^{\ast})a=a^{n}r^{\ast}a\in a^{n}R.
\end{split}
\end{equation*}
Thus $a\in Ra^{\ast}a\cap a^{n}R$.

$(2)\Rightarrow (1).$ Suppose $a\in Ra^{\ast}a\cap a^{n}R$,
there exist $s,t\in R$ such that $a=sa^{\ast}a=a^{n}t$.
Thus we get
\begin{equation*}
\begin{split}
a=sa^{\ast}a=s(a^{n}t)^{\ast}a=st^{\ast}(a^{\ast})^{n}a\in R(a^{\ast})^{n}a.
\end{split}
\end{equation*}

$(1)\Rightarrow (3).$ Assume that $a=s(a^{\ast})^{n}a$ for some $s\in R$,
which gives that $1-s(a^{\ast})^{n}\in$ $^{\circ}\!a$.
Since we can write $r$ as $r=r(1-s(a^{\ast})^{n})+rs(a^{\ast})^{n}$ for any $r\in R$,
where $r(1-s(a^{\ast})^{n})\in$ $^{\circ}\!a$ and $rs(a^{\ast})^{n}\in R(a^{\ast})^{n}$,
thus $R=$ $^{\circ}\!a+R(a^{\ast})^{n}$.

If $x\in R(a^{\ast})^{n}\cap$ $^{\circ}\!a$,
then $xa=0$ and $x=t(a^{\ast})^{n}$ for some $t\in R$.
Moreover,
\begin{equation*}
\begin{split}
    x
    &~=t(a^{\ast})^{n}=t(a^{\ast})^{n-1}a^{\ast}=t(a^{\ast})^{n-1}(s(a^{\ast})^{n}a)^{\ast}\\
    &~=t(a^{\ast})^{n}a^{n}s^{\ast}=xa^{n}s^{\ast}=0.
\end{split}
\end{equation*}
Hence $R=$ $^{\circ}\!a\oplus R(a^{\ast})^{n}$.

$(3)\Rightarrow (4)$ is trivial.

$(4)\Rightarrow (1).$ Since $a\in Ra=($ $^{\circ}\!a+R(a^{\ast})^{n})a\subseteq R(a^{\ast})^{n}a$,
which gives the condition $(1)$.

(II) The proof is similar to the proof of (I).
\end{proof}

Using Proposition~\ref{a**a},
we obtain the following new characterizations of core and dual core inverses which will be useful in the upcoming results.

\begin{thm}\label{core-inverse 2}
Let $a\in R$,
$n\geqslant 2$,
we have the following results:\\
$(1)$ $a\in R^{\co}$ if and only if $a\in R(a^{\ast})^{n}a\cap Ra^{n}$.
In this case,
$a^{\co}=a^{n-1}s^{\ast}$ for some $s\in R$ such that $a=s(a^{\ast})^{n}a$; \\
$(2)$ $a\in R_{\co}$ if and only if $a\in a(a^{\ast})^{n}R\cap a^{n}R$.
In this case,
$a_{\co}=t^{\ast}a^{n-1}$ for some $t\in R$ such that $a=a(a^{\ast})^{n}t$. \\
\end{thm}

\begin{proof}
$(1)$ Since $a\in R^{\#}$ if and only if $a\in a^{2}R\cap Ra^{2}$ by Lemma~\ref{group-inverse},
thus $a=a^{2}x=ya^{2}$ for some $x, y\in R$,
further we have
\begin{equation*}
\begin{split}
    a=a^{2}x=a(a^{2}x)x=a^{3}x=\cdots =a^{n}x\in a^{n}R
\end{split}
\end{equation*}
and
\begin{equation*}
\begin{split}
    a=ya^{2}=y(ya^{2})a=ya^{3}=\cdots =ya^{n}\in Ra^{n},
\end{split}
\end{equation*}
where $n\geqslant 2$.
Hence it is easy to deduce that $a\in R^{\#}$ if and only if $a\in a^{n}R\cap Ra^{n}$ for $n\geqslant 2$.

Applying Proposition~\ref{a**a},
$a\in R(a^{\ast})^{n}a\cap Ra^{n}$ if and only if $a\in Ra^{\ast}a\cap a^{n}R\cap Ra^{n}$,
which shows that $a\in R(a^{\ast})^{n}a\cap Ra^{n}$ if and only if $a\in Ra^{\ast}a\cap R^{\#}$.
Since $a\in Ra^{\ast}a$ is equivalent to $a\in R^{\{1,3\}}$ by Lemma~\ref{13-14-inverse},
thus $a\in R(a^{\ast})^{n}a\cap Ra^{n}$ if and only if $a\in R^{\co}$ by Lemma~\ref{dual-core-inverse}.

Next,
we give the representation of $a^{\co}$.
Since $a\in R(a^{\ast})^{n}a\cap Ra^{n}$,
there exists $s\in R$ such that $a=s(a^{\ast})^{n}a$.
By Lemma~\ref{13-14-inverse},
we have
\begin{equation*}
\begin{split}
    (s(a^{\ast})^{n-1})^{\ast}=a^{n-1}s^{\ast}\in a\{1,3\}.
\end{split}
\end{equation*}
Using Lemma~\ref{dual-core-inverse},
we obtain
\begin{equation*}
\begin{split}
    a^{\co}=a^{\#}aa^{(1,3)}=a^{\#}a(a^{n-1}s^{\ast})=a^{n-1}s^{\ast}.
\end{split}
\end{equation*}

$(2)$ Similarly as $(1)$.
\end{proof}

It is easy to see that Theorem~\ref{core-inverse 2} is also true in a semigroup by the proof of it.
From Lemma~\ref{dual-core-inverse},
we can obtain that $a\in R^{\co}\cap R_{\co}$ if and only if $a\in R^{\dagger}\cap R^{\#}$.
Therefore,
we have the following result by applying Proposition~\ref{a**a} and Theorem~\ref{core-inverse 2}.

\begin{thm}\label{cap1}
Let $a\in R$,
$n\geqslant 2$.
The following conditions are equivalent:\\
$(1)$ $a\in R^{\dagger}\cap R^{\#}$;\\
$(2)$ $a\in R^{\co}\cap R_{\co}$;\\
$(3)$ $a\in a(a^{\ast})^{n}R\cap R(a^{\ast})^{n}a$;\\
$(4)$ $R=$ $^{\circ}\!a\oplus R(a^{\ast})^{n}$, $R=a^{\circ}\oplus (a^{\ast})^{n}R$;\\
$(5)$ $R=$ $^{\circ}\!a+R(a^{\ast})^{n}$, $R=a^{\circ}+(a^{\ast})^{n}R$;\\
$(6)$ $R=$ $^{\circ}\!a\oplus R(a^{\ast})^{n}$, $R=a^{\circ}+(a^{\ast})^{n}R$;\\
$(7)$ $R=$ $^{\circ}\!a+R(a^{\ast})^{n}$, $R=a^{\circ}\oplus (a^{\ast})^{n}R$.\\
In this case,
\begin{equation*}
\begin{split}
    a^{\co}
    &~= a^{n-1}s^{\ast},\\
    a_{\co}
    &~= t^{\ast}a^{n-1},\\
    a^{\dagger}
    &~=t^{\ast}a^{2n-1}s^{\ast},\\
    a^{\#}
    &~=(a^{n-1}s^{\ast})^{2}a=a(t^{\ast}a^{n-1})^{2},
\end{split}
\end{equation*}
where $a= s(a^{\ast})^{n}a= a(a^{\ast})^{n}t$ for some $s,t\in R$.
\end{thm}
\begin{proof}
The equivalences of seven conditions above and the representations of $a^{\co}$ and $a_{\co}$ can be easily obtained by Proposition~\ref{a**a} and Theorem~\ref{core-inverse 2}.
We will give the representations of $a^{\dagger}$ and $a^{\#}$ in the following.

Suppose $a= s(a^{\ast})^{n}a= a(a^{\ast})^{n}t$ for some $s, t\in R$,
so $a^{n-1}s^{\ast}\in a\{1,3\}$ and $t^{\ast}a^{n-1}\in a\{1,4\}$ follow from Lemma~\ref{13-14-inverse}.
Applying Lemma~\ref{M-P1} and Lemma~\ref{core->group},
we get
\begin{equation*}
\begin{split}
    a^{\dagger}
    &~=a^{(1,4)}aa^{(1,3)}=(t^{\ast}a^{n-1})a(a^{n-1}s^{\ast})=t^{\ast}a^{2n-1}s^{\ast},\\
    a^{\#}
    &~=(a^{\co})^{2}a=(a^{n-1}s^{\ast})^{2}a=a(a_{\co})^{2}=a(t^{\ast}a^{n-1})^{2}.
\end{split}
\end{equation*}
\end{proof}

\begin{rem}
When taking $n=1$,
each of these five conditions $(3)$, $(4)$, $(5)$, $(6)$ and $(7)$ in Theorem~\ref{cap1} is equivalent to $a\in R^{\dagger}$ $($see \cite{Han,P}$)$.
\end{rem}

\section{Characterizing core (or dual core) inverses by Hermitian elements or projections in a ring}\label{a}
In this section,
we present some new equivalent conditions for the existence of core inverses.
Before we start,
look at the following two known results.

\begin{thm}\cite[Theorem $3.6$ and $3.7$]{XSZ2}
Let $a\in R$ and $n\geqslant 1$,
the following conditions are equivalent: \\
$(1)$ $a\in R^{\dagger}$;\\
$(2)$ there exists a projection (or Hermitian element) $p$ such that $pa=0, u=(aa^{\ast})^{n}+p\in R^{-1}$;\\
$(3)$ there exists a projection (or Hermitian element) $q$ such that $aq=0, v=(a^{\ast}a)^{n}+q\in R^{-1}$.\\
In this case,
\begin{equation*}
\begin{split}
a^{\dagger}=a^{\ast}u^{-1}(aa^{\ast})^{2n-1}(u^{-1})^{\ast}=(v^{-1})^{\ast}(a^{\ast}a)^{2n-1}v^{-1}a^{\ast}.
\end{split}
\end{equation*}
\end{thm}

\begin{thm}\cite[Proposition $8.24$]{KPS}
Let $a\in R$,
the following conditions are equivalent: \\
$(1)$ $a\in R^{\#}$;\\
$(2)$ there exists an idempotent $p$ such that $pa=ap=0, u=a+p\in R^{-1}$.\\
In this case,
\begin{equation*}
\begin{split}
a^{\#}=u^{-1}(1-p).
\end{split}
\end{equation*}
\end{thm}

Inspired by this two theorems,
we extend the same reasoning to the core and dual core inverses.
Moreover,
we give their representations.

\begin{thm} \label{n-core-inverse}
Let $a\in R$,
$n\geqslant 2$.
The following conditions are equivalent: \\
$(1)$ $a\in R^{\co}$; \\
$(2)$ there exists a unique projection $p$ such that $pa=0$, $u=a^{n}+p\in R^{-1}$;\\
$(3)$ there exists a Hermitian element $p$ such that $pa=0$, $u=a^{n}+p\in R^{-1}$.\\
In this case,
\begin{equation*}
\begin{split}
    a^{\co}=a^{n-1}u^{-1}.
\end{split}
\end{equation*}
\end{thm}

\begin{proof}
$(1)\Rightarrow (2).$ Let $p=1-aa^{\co}$,
we observe first that $p$ is a projection satisfying $pa=0$.
It is necessary for us to show that $^{\circ}\!(a^{n})=$ $^{\circ}\!(1-p)$.
If $xa^{n}=0$,
then
\begin{equation*}
\begin{split}
    0=xa^{n}=x(1-p)a^{n}=x(1-p)(a^{n}+p).
\end{split}
\end{equation*}
By $a^{n}+p\in R^{-1}$,
we have $x(1-p)=0$.
Conversely,
if $y(1-p)=0$,
then $ya^{n}=y(1-p)a^{n}=0$.
Thus $^{\circ}\!(a^{n})=$ $^{\circ}\!(1-p)$.
Assume that $p, q$ are both projection which satisfy the condition $(2)$,
then $^{\circ}\!(1-p)=$ $^{\circ}\!(a^{n})=$ $^{\circ}\!(1-q)$.
By $p\in$ $^{\circ}\!(1-p)=$ $^{\circ}\!(1-q)$,
we obtain $p=pq$.
Similarly,
we can get $q=qp$ from $q\in$ $^{\circ}\!(1-q)=$ $^{\circ}\!(1-p)$.
Thus
\begin{equation*}
\begin{split}
    p=p^{\ast}=(pq)^{\ast}=q^{\ast}p^{\ast}=qp=q.
\end{split}
\end{equation*}

Next,
we prove the invertibility of $u$ by induction on $n$.

When $n=2$,
it is easy to verify that
\begin{equation*}
\begin{split}
    (a+1-aa^{\co})(a^{\co}+1-a^{\co}a)=1=(a^{\co}+1-a^{\co}a)(a+1-aa^{\co}),
\end{split}
\end{equation*}
thus $a+1-aa^{\co}=1+aa^{\co}(a-1)$ is invertible.
Moreover,
$1+(a-1)aa^{\co}=a^{2}a^{\co}+1-aa^{\co}$ is invertible by Lemma \ref{J.L}.
Therefore,
$a^{2}+1-aa^{\co}=(a^{2}a^{\co}+1-aa^{\co})(a+1-aa^{\co})$ is invertible.

We assume that $n>2$ and the result is true for the case $n-1$.
By assumption,
$a\in R^{\co}$ implies that $a^{n-1}+1-aa^{\co}$ is invertible.
Hence,
$a^{n}+1-aa^{\co}=(a^{2}a^{\co}+1-aa^{\co})(a^{n-1}+1-aa^{\co})$ is invertible.

$(2)\Rightarrow (3)$ is trivial.

$(3)\Rightarrow (1).$ Assume that $u=a^{n}+p\in R^{-1}$,
where $p=p^{\ast}, n\geqslant 2$,
and then $u^{\ast}=(a^{\ast})^{n}+p$ is also invertible.

Since $ua=a^{n+1}$ and $u^{\ast}a=(a^{\ast})^{n}a$,
we obtain
\begin{equation*}
\begin{split}
    a=u^{-1}a^{n+1}\in Ra^{n}
\end{split}
\end{equation*}
and
\begin{equation*}
\begin{split}
    a=(u^{\ast})^{-1}(a^{\ast})^{n}a\in R(a^{\ast})^{n}a.
\end{split}
\end{equation*}
Thus $a\in R(a^{\ast})^{n}a\cap Ra^{n}$,
and then we have $a\in R^{\co}$ and $a^{\co}=a^{n-1}u^{-1}$ by Theorem~\ref{core-inverse 2}.
\end{proof}

There is a corresponding result for dual core inverses of $a\in R_{\co}$.
The following theorem shows that Theorem~\ref{n-core-inverse} is true when taking $n=1$,
but its proof is different from the proof of Theorem~\ref{n-core-inverse},
and so is the expression of the core inverse of $a$.

\begin{thm} \label{n=1-core-inverse}
Let $a\in R$.
The following conditions are equivalent: \\
$(1)$ $a\in R^{\co}$; \\
$(2)$ there exists a unique projection $p$ such that $pa=0$, $u=a+p\in R^{-1}$;\\
$(3)$ there exists a Hermitian element $p$ such that $pa=0$, $u=a+p\in R^{-1}$.\\
In this case,
\begin{equation*}
\begin{split}
    a^{\co}=u^{-1}au^{-1}=(u^{\ast}u)^{-1}a^{\ast}.
\end{split}
\end{equation*}
\end{thm}

\begin{proof}
$(1)\Rightarrow (2).$ Let $p=1-aa^{\co}$,
$p$ is a projection satisfying $pa=0$,
and the proof of the uniqueness of $p$ is similar to Theorem~\ref{n-core-inverse}.
It is easy to verify
\begin{equation*}
\begin{split}
    (a+1-aa^{\co})(a^{\co}+1-a^{\co}a)=1=(a^{\co}+1-a^{\co}a)(a+1-aa^{\co}).
\end{split}
\end{equation*}
Thus $a+1-aa^{\co}$ is invertible.

$(2)\Rightarrow (3).$ Obviously.

$(3)\Rightarrow (1).$ Assume that $u=a+p\in R^{-1}$,
where $p=p^{\ast}$,
and then $u^{\ast}=a^{\ast}+p$ is also invertible.

Since $ua=a^{2}$ and $u^{\ast}a=a^{\ast}a$,
we obtain
\begin{equation} \label{n=1-core-inverse 1}
\begin{split}
    a=u^{-1}a^{2}\in Ra^{2}
\end{split}
\end{equation}
and
\begin{equation} \label{n=1-core-inverse 2}
\begin{split}
    a=(u^{\ast})^{-1}a^{\ast}a\in Ra^{\ast}a.
\end{split}
\end{equation}
By Lemma~\ref{13-14-inverse},
it is easily seen that $a\in R^{\{1,3\}}$ with $u^{-1}\in a\{1,3\}$ from the equation (\ref{n=1-core-inverse 2}),
so we have
\begin{equation} \label{n=1-core-inverse 3}
\begin{split}
    a=aa^{(1,3)}a=au^{-1}a.
\end{split}
\end{equation}
Moreover,
$pu=p(a+p)=p^{2}$ implies $p=p^{2}u^{-1}$,
thus we have
\begin{equation} \label{n=1-core-inverse 4}
\begin{split}
    p^{2}u^{-1}=p=p^{\ast}=(u^{\ast})^{-1}p^{2},
\end{split}
\end{equation}
direct calculations with the use of (\ref{n=1-core-inverse 4}) show that
\begin{equation} \label{n=1-core-inverse 5}
\begin{split}
    pu^{-2}
    &~=((u^{\ast})^{-1}p^{2})u^{-2}=(u^{\ast})^{-1}(p^{2}u^{-1})u^{-1}\\
    &~=(u^{\ast})^{-1}((u^{\ast})^{-1}p^{2})u^{-1}=(u^{\ast})^{-2}p^{2}u^{-1}\\
    &~=(u^{\ast})^{-2}p.
\end{split}
\end{equation}
Therefore
\begin{equation} \label{n=1-core-inverse 6}
\begin{split}
    u^{-1}
    &~=uu^{-2}=(a+p)u^{-2}=au^{-2}+pu^{-2}\\
    &~\stackrel{(\ref{n=1-core-inverse 5})}{=}au^{-2}+(u^{\ast})^{-2}p.
\end{split}
\end{equation}
Thus we obtain
\begin{equation} \label{n=1-core-inverse 7}
\begin{split}
    a
    &~\stackrel{(\ref{n=1-core-inverse 3})}{=}au^{-1}a\stackrel{(\ref{n=1-core-inverse 6})}{=}a(au^{-2}+(u^{\ast})^{-2}p)a\\
    &~= a^{2}u^{-2}a\in a^{2}R.
\end{split}
\end{equation}
The equations (\ref{n=1-core-inverse 1}),
(\ref{n=1-core-inverse 2}) and (\ref{n=1-core-inverse 7}) lead to $a\in R^{\co}$ by Lemma~\ref{dual-core-inverse}.

Applying Lemma~\ref{group-inverse},
we obtain $a^{\#}=u^{-2}a$,
again by Lemma~\ref{dual-core-inverse},
we have
\begin{equation*}
\begin{split}
    a^{\co}=a^{\#}aa^{(1,3)}=(u^{-2}a)au^{-1}=u^{-1}(u^{-1}a^{2})u^{-1}\stackrel{(\ref{n=1-core-inverse 1})}{=}u^{-1}au^{-1}.
\end{split}
\end{equation*}
Further,
since $u^{-1}\in a\{1,3\}$,
thus
\begin{equation*}
\begin{split}
    a^{\co}=u^{-1}au^{-1}=u^{-1}(au^{-1})^{\ast}=u^{-1}(u^{\ast})^{-1}a^{\ast}=(u^{\ast}u)^{-1}a^{\ast}.
\end{split}
\end{equation*}
\end{proof}

The analogous result for dual core inverses of $a\in R_{\co}$ is valid.

\begin{rem} \label{3.9}
Theorem~\ref{n-core-inverse} and Theorem~\ref{n=1-core-inverse} show that $a\in R^{\co}$ if and only if
there exists $p=p^{\ast}(=p^{2})$ such that $pa=0$, $u=a^{n}+p\in R^{-1}$ for all choices $n\geqslant 1$.
Dually,
$a\in R_{\co}$ if and only if
there exists $q=q^{\ast}(=q^{2})$ such that $aq=0$, $u=a^{n}+q\in R^{-1}$ for all choices $n\geqslant 1$.
\end{rem}

Under the condition $(2)$ of Theorem~\ref{n=1-core-inverse},
since $u^{\ast}u=a^{\ast}a+p$,
the expression of the core inverse of $a$ can be showed as $a^{\co}=(a^{\ast}a+p)^{-1}a^{\ast}$.
Through this expression,
we naturally want to know whether $a$ is core invertible or not when there is a unique projection $p$ such that $pa=0, a^{\ast}a+p\in R^{-1}$.
However,
it is not true.
Here is a counterexample.

\begin{eg}\label{eg1}
Let $R$ be an infinite matrix ring over complex filed whose rows and columns are both finite,
let conjugate transpose be the involution and $a=\sum\limits_{i=1}^{\infty}e_{i+1,i}$.
Then $a^{\ast}a=1$,
$aa^{\ast}=\sum\limits_{i=2}^{\infty}e_{i,i}$.
Set $p=0$,
then it is a projection satisfying $pa=0$ and $a^{\ast}a+p=1\in R^{-1}$.
But $a$ is not group invertible,
thus $a$ is not core invertible.
\end{eg}

This counterexample shows that even if there is a unique projection $p$ such that $pa=0$ and $a^{\ast}a+p\in R^{-1}$,
$a$ is not necessary to be core invertible in general rings.
However,
it is true when we take $R$ as a Dedekind-finite ring which satisfies the property that $ab=1$ implies $ba=1$ for any $a, b\in R$,
see the following theorem.

\begin{thm} \label{D-finite ring}
Let $R$ be a Dedekind-finite ring and $a\in R$.
The following conditions are equivalent: \\
$(1)$ $a\in R^{\co}$; \\
$(2)$ there exists a unique projection $p$ such that $pa=0$, $a^{\ast}a+p\in R^{-1}$;\\
$(3)$ there exists a unique projection $p$ such that $pa=0$, $a^{\ast}a+p$ is right invertible; \\
$(4)$ there exists a unique projection $p$ such that $pa=0$, $a^{\ast}a+p$ is left invertible.\\
In this case,
\begin{equation*}
\begin{split}
    a^{\co}=(a^{\ast}a+p)^{-1}a^{\ast}.
\end{split}
\end{equation*}
\end{thm}
\begin{proof}
Since $a^{\ast}a+p$ is Hermitian,
thus $a^{\ast}a+p$ is one-sided invertible if and only if it is invertible,
hence the conditions $(2)$, $(3)$ and $(4)$ are equivalent.
Next,
we mainly show the equivalence between the conditions $(1)$ and $(2)$.

$(1)\Rightarrow (2).$ Assume that $a\in R^{\co}$ and let $p=1-aa^{\co}$,
we have $a+p\in R^{-1}$ by Theorem~\ref{n=1-core-inverse},
thus $(a+p)^{\ast}=a^{\ast}+p\in R^{-1}$.
Therefore $a^{\ast}a+p=(a^{\ast}+p)(a+p)$ is invertible.

$(2)\Rightarrow (1).$ Let $u=a+p$,
$u^{\ast}u=a^{\ast}a+p\in R^{-1}$.
As $R$ is a Dedekind-finite ring,
thus $u\in R^{-1}$,
which guarantees $a\in R^{\co}$ and $a^{\co}=(a^{\ast}a+p)^{-1}a^{\ast}$ by Theorem~\ref{n=1-core-inverse}.
\end{proof}

As mentioned before,
$a\in R^{\co}\cap R_{\co}$ if and only if $a\in R^{\dagger}\cap R^{\#}$,
and an element $a$ is called an $\mathrm{EP}$ element if $a\in R^{\dagger}\cap R^{\#}$ with $a^{\dagger}=a^{\#}$,
thus it is obvious that $a\in R^{\mathrm{EP}}$ if and only if $a\in R^{\co}\cap R_{\co}$ and $a^{\co}=a_{\co}$.
By Remark~\ref{3.9},
we obtain the following theorem.

\begin{thm} \label{EP}
Let $a\in R$,
$n\geqslant 1$.
The following conditions are equivalent: \\
$(1)$ $a\in R^{\mathrm{EP}}$; \\
$(2)$ there exists a unique projection $p$ such that $pa=ap=0$, $a^{n}+p\in R^{-1}$;\\
$(3)$ there exists a Hermitian element $p$ such that $pa=ap=0$, $a^{n}+p\in R^{-1}$.
\end{thm}
\begin{proof}
$(1)\Rightarrow (2).$ Suppose $a\in R^{\mathrm{EP}}$,
so $a\in R^{\dagger}\cap R^{\#}$ and $a^{\dagger}=a^{\#}$.
Let $p=1-a^{\#}a=1-a^{\dagger}a$,
it is easily seen that $p$ is a projection satisfying $pa=ap=0$.
Since
\begin{equation*}
\begin{split}
    (a^{n}+1-a^{\#}a)((a^{\#})^{n}+1-a^{\#}a)=1=((a^{\#})^{n}+1-a^{\#}a)(a^{n}+1-a^{\#}a),
\end{split}
\end{equation*}
thus $a^{n}+1-a^{\#}a$ is invertible.

$(2)\Rightarrow (3).$ Obviously.

$(3)\Rightarrow (1).$ We can see that $a\in R^{\co}\cap R_{\co}$ follows from Remark~\ref{3.9},
so we only need to show that $a^{\co}=a_{\co}$.
Write $u=a^{n}+p$,
when taking $n=1$,
we have $a^{\co}=u^{-1}au^{-1}=a_{\co}$ by Theorem \ref{n=1-core-inverse}.
When $n\geqslant 2$,
according to Theorem \ref{n-core-inverse},
$a^{\co}=a^{n-1}u^{-1}, a_{\co}=u^{-1}a^{n-1}$.
Since $pa=ap=0$ implies $au=ua$,
which obtain that $a$ commutes with $u^{-1}$,
thus $a^{\co}=a^{n-1}u^{-1}=u^{-1}a^{n-1}=a_{\co}$.
Therefore,
$a\in R^{\mathrm{EP}}$.
\end{proof}

\section{Characterizing core and dual core inverse of a regular element by units in a ring}
In \cite[Theorem $5.6$]{Chen},
J.L. Chen et al. characterized core and dual core inverse of a regular element by units in a ring,
that is to say,
they proved that $a\in R^{\co}\cap R_{\co}$ if and only if $a^{2}a^{\ast}+1-aa^{-}$ is invertible.
In this section,
we will show readers that the result is valid when we change the quadratic component $a^{2}a^{\ast}$ to $a(a^{\ast})^{2}$.
Moreover,
we give the expressions of $a^{\dagger},a^{\#},a^{\co},a_{\co}$.

\begin{thm} \label{aa**aR}
Let $a\in R$,
$n\geqslant 2$.
If $a\in R$ is regular with $a^{-}\in a{\{1\}}$,
then the following conditions are equivalent:\\
$(1)$ $a\in R^{\dagger}\cap R^{\#}$;\\
$(2)$ $a\in R^{\co}\cap R_{\co}$;\\
$(3)$ $u=(a^{\ast})^{n}a+1-a^{-}a$ is invertible;\\
$(4)$ $v=a(a^{\ast})^{n}+1-aa^{-}$ is invertible;\\
$(5)$ $s=a^{-}a(a^{\ast})^{n}a+1-a^{-}a$ is invertible;\\
$(6)$ $t=a(a^{\ast})^{n}aa^{-}+1-aa^{-}$ is invertible.\\
In this case,
\begin{equation*}
\begin{split}
    a^{\co}
    &~=a^{n-1}(v^{-1}a)^{\ast},\\
    a_{\co}
    &~=(au^{-1})^{\ast}a^{n-1},\\
    a^{\dagger}
    &~=(au^{-1})^{\ast}a^{2n-1}(v^{-1}a)^{\ast},\\
    a^{\#}
    &~=(a^{n-1}(v^{-1}a)^{\ast})^{2}a.
\end{split}
\end{equation*}
\end{thm}

\begin{proof}
It is obvious that the conditions $(1)$ and $(2)$ are equivalent,
and equivalences of the conditions $(3), (4), (5)$ and $(6)$ can be deduced from Lemma~\ref{J.L}.

$(2)\Rightarrow (3).$ Since
\begin{equation*}
\begin{split}
    &~ ~~~u[a^{-}aa^{\co}(a_{\co}^{n})^{\ast}+1-a_{\co}a]\\
    &~=[(a^{\ast})^{n}a+1-a^{-}a][a^{-}aa^{\co}(a_{\co}^{n})^{\ast}+1-a_{\co}a]\\
    &~=(a^{\ast})^{n}aa^{-}aa^{\co}(a_{\co}^{n})^{\ast}+(1-a^{-}a)(1-a_{\co}a)\\
    &~=(a^{\ast})^{n}aa^{\co}(a_{\co}^{n})^{\ast}+1-a_{\co}a\\
    &~=(a^{\ast})^{n}(aa^{\co})^{\ast}(a_{\co}^{n})^{\ast}+1-a_{\co}a\\
    &~=(aa^{\co}a^{n})^{\ast}(a_{\co}^{n})^{\ast}+1-a_{\co}a\\
    &~=(a^{n})^{\ast}(a_{\co}^{n})^{\ast}+1-a_{\co}a\\
    &~=(a_{\co}^{n}a^{n})^{\ast}+1-a_{\co}a\\
    &~=(a_{\co}a)^{\ast}+1-a_{\co}a\\
    &~=a_{\co}a+1-a_{\co}a\\
    &~=1,
\end{split}
\end{equation*}
where the eighth equation follows from
\begin{equation*}
\begin{split}
    a_{\co}^{n}a^{n}
    &~=a_{\co}^{n-2}(a_{\co}^{2}a)a^{n-1}=a_{\co}^{n-2}a_{\co}a^{n-1}=a_{\co}^{n-1}a^{n-1}\\
    &~=a_{\co}^{n-3}(a_{\co}^{2}a)a^{n-2}=a_{\co}^{n-3}a_{\co}a^{n-2}=a_{\co}^{n-2}a^{n-2}\\
    &~=\cdots =a_{\co}^{2}a^{2}=(a_{\co}^{2}a)a=a_{\co}a,
\end{split}
\end{equation*}
thus $u$ is right invertible with right inverse $a^{-}aa^{\co}(a_{\co}^{n})^{\ast}+1-a_{\co}a$.
Similarly,
$[((a^{\co})^{n})^{\ast}a_{\co}aa^{-}+1-aa^{\co}]v=1$ implies that $v$ is left invertible,
thus $u$ is left invertible by Lemma~\ref{J.L}.
Further,
$u=(a^{\ast})^{n}a+1-a^{-}a$ is invertible.

$(3)\Rightarrow (2).$ Since $u$ is invertible if and only if $v$ is invertible follows from Lemma~\ref{J.L},
$au=a(a^{\ast})^{n}a=va$,
so we have
\begin{equation*}
\begin{split}
    a=a(a^{\ast})^{n}au^{-1}\in a(a^{\ast})^{n}R,\\
    a=v^{-1}a(a^{\ast})^{n}a\in R(a^{\ast})^{n}a.
\end{split}
\end{equation*}
By Theorem~\ref{cap1},
we obtain $a\in R^{\co}\cap R_{\co}$ and the following representations,
\begin{equation*}
\begin{split}
    a^{\co}
    &~=a^{n-1}(v^{-1}a)^{\ast},\\
    a_{\co}
    &~=(au^{-1})^{\ast}a^{n-1},\\
    a^{\dagger}
    &~=(au^{-1})^{\ast}a^{2n-1}(v^{-1}a)^{\ast},\\
    a^{\#}
    &~=[a^{n-1}(v^{-1}a)^{\ast}]^{2}a.
\end{split}
\end{equation*}
\end{proof}

In \cite[Theorem $3.2$]{X1},
X. Mary et al. proved that $m$ is invertible along $d$ if and only if $d\in dmdR\cap Rdmd$.
Comparing Theorem~\ref{aa**aR} and Lemma~\ref{Mary},
we have the following corollary.

\begin{cor} \label{cap2}
Let $a\in R$,
$n\geqslant 2$.
The following conditions are equivalent:\\
$(1)$ $a\in R^{\dagger}\cap R^{\#}$;\\
$(2)$ $a\in R^{\co}\cap R_{\co}$;\\
$(3)$ $(a^{\ast})^{n}$ is invertible along $a$;\\
$(4)$ $a\in a(a^{\ast})^{n}aR\cap Ra(a^{\ast})^{n}a$.\\
In this case,
\begin{equation*}
\begin{split}
    a^{\co}
    &~= a^{n-1}a^{\ast}y^{\ast},\\
    a_{\co}
    &~= x^{\ast}a^{\ast}a^{n-1},\\
    a^{\dagger}
    &~=x^{\ast}a^{\ast}a^{2n-1}a^{\ast}y^{\ast},\\
    a^{\#}
    &~=[a^{n-1}a^{\ast}y^{\ast}]^{2}a=a[x^{\ast}a^{\ast}a^{n-1}]^{2},
\end{split}
\end{equation*}
where $a= a(a^{\ast})^{n}ax= ya(a^{\ast})^{n}a$ for some $x,y\in R$.
\end{cor}

Theorem~\ref{cap1} and Corollary~\ref{cap2} show that $a\in R^{\co}\cap R_{\co}$ if and only if $a\in a(a^{\ast})^{n}R\cap R(a^{\ast})^{n}a$ if and only if $a\in a(a^{\ast})^{n}aR\cap Ra(a^{\ast})^{n}a$.
And it is easily seen that $a\in Ra(a^{\ast})^{n}a\subseteq R(a^{\ast})^{n}a$,
but $a\in R(a^{\ast})^{n}a$ does not lead to $a\in Ra(a^{\ast})^{n}a$.
There is a counterexample in the following when taking $n=2$.

\begin{eg}
Let $R=M_{2}(\mathbb{C})$ be the ring of all $2 \times 2$ matrices over the complex field $\mathbb{C}$.
Taking transposition as involution,
considering the matrix $a=\left(
                            \begin{array}{cc}
                              1 & i \\
                              0 & 0  \\
                            \end{array}
                          \right)$,
we have $a^{2}=a, aa^{\ast}=0, a^{\ast}a=\left(
               \begin{array}{cc}
                1 & i \\
                i & -1  \\
                \end{array}
                \right)
\neq 0$.
Thus $a=\left(
               \begin{array}{cc}
                1 & 0 \\
                0 & 0  \\
                \end{array}
                \right)
(a^{\ast})^{2}a\in R(a^{\ast})^{2}a$,
but $a\notin Ra(a^{\ast})^{2}a$.
\end{eg}

In \cite{CZ},
H.H. Zhu et al. proved $a\in Raa^{\ast}a$ if and only if $a\in aa^{\ast}aR$.
Reproducing his work,
we get the following result.

\begin{prop} \label{aa*na}
Let $a\in R$,
$n\geqslant 1$,
we have the following results:\\
$(1)$ if $a\in Ra(a^{\ast})^{n}a$, then $a\in a^{n}a^{\ast}a^{n}R$;\\
$(2)$ if $a\in a(a^{\ast})^{n}aR$, then $a\in Ra^{n}a^{\ast}a^{n}$.
\end{prop}

\begin{proof}
$(1)$ Suppose $a\in Ra(a^{\ast})^{n}a$,
there exists $x\in R$ such that
\begin{equation} \label{aa*na1}
\begin{split}
    a=xa(a^{\ast})^{n}a.
\end{split}
\end{equation}
Taking involution on (\ref{aa*na1}),
we get
\begin{equation} \label{aa*na2}
\begin{split}
    a^{\ast}=a^{\ast}a^{n}a^{\ast}x^{\ast}.
\end{split}
\end{equation}
Again by (\ref{aa*na1}),
we obtain
\begin{equation*}
\begin{split}
    [xa(a^{\ast})^{n}]^{\ast}
    &~=a^{n}(xa)^{\ast}=aa^{n-1}(xa)^{\ast}\\
    &~\stackrel{(\ref{aa*na1})}{=}[xa(a^{\ast})^{n}a]a^{n-1}(xa)^{\ast}\\
    &~=xa(a^{\ast})^{n}a^{n}(xa)^{\ast}.
\end{split}
\end{equation*}
Hence $[xa(a^{\ast})^{n}]^{\ast}$ is symmetric,
that is to say,
\begin{equation} \label{aa*na3}
\begin{split}
    xa(a^{\ast})^{n}=[xa(a^{\ast})^{n}]^{\ast}.
\end{split}
\end{equation}
The equalities (\ref{aa*na1}),
(\ref{aa*na2}) and (\ref{aa*na3}) give
\begin{equation*}
\begin{split}
    a
    &~\stackrel{(\ref{aa*na1})}{=}xa(a^{\ast})^{n}a \stackrel{(\ref{aa*na3})}{=}[xa(a^{\ast})^{n}]^{\ast}a\\
    &~=a^{n}a^{\ast}x^{\ast}a \stackrel{(\ref{aa*na2})}{=}a^{n}(a^{\ast}a^{n}a^{\ast}x^{\ast})x^{\ast}a\\
    &~=a^{n}a^{\ast}a^{n}(a^{\ast}x^{\ast}x^{\ast}a)\in a^{n}a^{\ast}a^{n}R.
\end{split}
\end{equation*}

$(2)$ Dually as $(1)$.
\end{proof}

As we all know,
$a\in Raa^{\ast}a$ if and only if $a\in aa^{\ast}aR$,
but $a\in Ra(a^{\ast})^{n}a$ is not equivalent to $a\in a(a^{\ast})^{n}aR$ when $n\geqslant 2$.
Taking the condition $n=2$ for example.

\begin{eg}
Let $R$ be as Example~\ref{eg1},
and let $a=\sum\limits_{i=1}^{\infty}e_{i,i+1}$.
Then $aa^{\ast}=1$,
$a^{\ast}a=\sum\limits_{i=2}^{\infty}e_{i,i}$.
Moreover,
$a(a^{\ast})^{2}a=a^{\ast}a$,
$a=aa^{\ast}a\in Ra(a^{\ast})^{2}a$,
but $a\notin a(a^{\ast})^{2}aR$.
However,
when taking $R$ as a Dedekind-finite ring,
we will get some unexpected results.
\end{eg}

\begin{thm}\label{D-f3}
Let $a\in R$,
$n\geqslant 2$,
consider the following conditions:\\
$(1)$ $R$ is a Dedekind-finite ring;\\
$(2)$ $a\in a(a^{\ast})^{n}aR$ if and only if $a\in Ra(a^{\ast})^{n}a$;\\
$(3)$ $aa^{\ast}=1$ implies $a^{\ast}a=1$ for any $a\in R$.\\
Then we have $(1)\Rightarrow (2)\Rightarrow (3)$.
\end{thm}
\begin{proof}
$(1)\Rightarrow (2).$ Since $R$ is a Dedekind-finite ring,
which guarantees that $(a^{\ast})^{n}a+1-a^{-}a$ is right invertible if and only if $(a^{\ast})^{n}a+1-a^{-}a$ is left invertible.
Hence,
$a\in a(a^{\ast})^{n}aR$ if and only if $a\in Ra(a^{\ast})^{n}a$ by Lemma~\ref{ZHH}.

$(2)\Rightarrow (3).$ Suppose that $aa^{\ast}=1$,
and then $a^{k}(a^{\ast})^{k}=1$ for any $k\geqslant 1$.
Moreover,
$a=a^{n-1}(a^{\ast})^{n-1}a\in R(a^{\ast})^{n-1}a=Ra(a^{\ast})^{n}a$,
which implies $a\in a(a^{\ast})^{n}aR=(a^{\ast})^{n-1}aR$ by the given condition $(2)$.
Thus there exists $t\in R$ such that $a=(a^{\ast})^{n-1}at$.
Furthermore,
\begin{equation*}
\begin{split}
    (a^{\ast})^{n-1}a^{n}
    &~=(a^{\ast})^{n-1}a^{n-1}((a^{\ast})^{n-1}at)\\
    &~=(a^{\ast})^{n-1}(a^{n-1}(a^{\ast})^{n-1})at\\
    &~=(a^{\ast})^{n-1}at=a,
\end{split}
\end{equation*}
hence we obtain
\begin{equation*}
\begin{split}
    a^{\ast}a
    &~=(a^{n-2}(a^{\ast})^{n-2})a^{\ast}a(a^{n-1}(a^{\ast})^{n-1})\\
    &~=a^{n-2}((a^{\ast})^{n-1}a^{n})(a^{\ast})^{n-1}=a^{n-2}a(a^{\ast})^{n-1}\\
    &~=a^{n-1}(a^{\ast})^{n-1}=1.
\end{split}
\end{equation*}
\end{proof}

\vspace{0.2cm} \noindent {\large\bf Acknowledgements}

This research is supported by the National Natural Science Foundation of China (No.11201063 and No.11371089); the Natural Science Foundation of Jiangsu Province (No.BK20141327).

\end{document}